\documentclass[12pt]{amsart}
\usepackage[colorlinks, backref]{hyperref}
\usepackage[T1]{fontenc}
\usepackage{amsmath, longtable}
\usepackage{amscd}
\usepackage{amssymb}
\usepackage{mathtools}
\let\SavedRightarrow=\Rightarrow
\usepackage{marvosym}
\let\Rightarrow=\SavedRightarrow

\newcommand{\Raa }{\mathcal R}

\newcommand{\Pee }{\mathcal P}

\newcommand{\pr}{\operatorname{pr}}
\newcommand{\supp}{\operatorname{supp}}

\renewcommand{\int}{\operatorname{Int}}

\newtheorem{theorem}{Theorem}

\theoremstyle{definition}
\newtheorem{example}{Example}

\author{Judyta B\k{a}k} 
\address{Judyta B\k{a}k\\  Institute of Mathematics, University of Silesia \\  Bankowa 14, 40-007 Katowice\\and Institute of Mathematics,
Jan Kochanowski University, 
\'{S}wi\k{e}tokrzyska~15,
25-406 Kielce, Poland} \email{jubak@us.edu.pl} 
\author{Andrzej Kucharski} 
\address{Andrzej Kucharski \\  Institute of Mathematics, University of Silesia \\  Bankowa 14, 40-007 Katowice} \email{andrzej.kucharski@us.edu.pl} 
\date{ver.30}
\begin{document}

\begin{abstract}
We prove that a player $\alpha$ has a winning strategy in the Banach--Mazur game on a space $X$ if and only if $X$ is F-Y countably $\pi$-domain representable. We show that Choquet complete spaces are F-Y countably domain representable. We give  an example of a space, which is F-Y countably domain representable, but it is not F-Y $\pi$-domain representable. 
\end{abstract}

\title{The Banach--Mazur game and the strong Choquet game in domain theory} 

\maketitle
 
\section{Introduction}


The famous Banach--Mazur game was invented by Mazur in 1935. For the history of game theory and facts  about game the reader is referred to the survey \cite{Telgarsky}. Let $X$ be  
a topological space and $X=A\cup B$ be any given  decomposition of $X$ into two disjoint sets. The game 
$BM(X,A,B)$ is played as follows: Two players, named $A$ and $B$ alternately choose open nonempty sets 
with $U_0\supseteq V_0\supseteq U_1\supseteq V_1\supseteq\cdots$. 

$\begin{array}{lccccr}

A\;&U_0&   &U_1&   &\\
 &   &   &   &   &\cdots\\
B\;&   &V_0&   &V_1& 
\end{array}  $

Player A wins this  game if $A\cap\bigcap_{n\in\omega}U_n\neq\emptyset$. Otherwise B wins.

We study well-known modification of this game considered by Choquet in 1958, known as Banach--Mazur game or Choquet game. Player $\alpha$ and $
\beta $  alternately choose open nonempty sets with $U_0\supseteq V_0\supseteq U_1\supseteq V_1\cdots$. 

$\begin{array}{lccccr}
\beta\;&U_0&   &U_1&   &\\
 &   &   &   &   &\cdots\\
\alpha\;&   &V_0&   &V_1& 
\end{array}  $

Player $\alpha$ wins this play if $\bigcap \{V_n: n\in \omega\}\neq\emptyset$. Otherwise $\beta$ wins. 
Denote this game by $BM(X).$ Every finite sequence of sets $(U_0,\ldots ,U_n)$, obtained by 
the first $n$ steps in this game is called \textit{legal $n$ moves} of $\beta$ (or partial play).  A \textit{strategy} for 
the player $\alpha$ in the game  $BM(X)$ is a map $s$ that assigns to each legal $n$ moves of $\beta$  a 
nonempty open set $V_n\subseteq U_n$. The strategy $s$ is called \textit{winning strategy} for player $
\alpha$ if for every legal $n$ moves of $\beta$ $(U_0,\ldots ,U_n)$ such that $V_{n}
=s(U_0,\ldots ,U_n)$, we have $\bigcap_{n=1}^{\infty} V_n\neq\emptyset$. The space $X$ is 
called \textit{weakly $\alpha$-favorable} (see \cite{White}) if $X$ admits a winning strategy for the 
player $\alpha$ in the game $BM(X)$. We say that $(W_0,\ldots,W_k)$ is \textit{stronger} than $(U_0,\ldots, U_m)$ if
$m\leq k$ and $U_0=W_0,\ldots,U_m=W_m$.  Notice that if $(W_0,\ldots,W_k)$ is stronger than $(U_0,\ldots, U_m)$, 
then $s(W_0,\ldots,W_k)\subseteq s(U_0,\ldots, U_m)$, we denote it by $(U_0,\ldots, U_m)\preceq (W_0,\ldots,W_k)$.

The \textit{strong Choquet} game is defined as follows: 

$\begin{array}{lccccr}
\beta\;&U_0\ni x_0&   &U_1\ni x_1&   &\\
 &   &   &   &   &\cdots\\
\alpha\;&   &V_0&   &V_1& 
\end{array}  $

Player $\beta$ and $\alpha$ take turns in playing nonempty open subset, similar to the Banach--Mazur game. In the first
 round, player $\beta$ starts by choosing a point $x_0$ and  an open set $U_0$ containing $x_0$, then player 
 $\alpha$ responds with an open set $V_0$ such that $x_0\in V_0\subseteq U_0$. In the $n$-th round, player $\beta$ 
 selects a point $x_n$ and an open set $U_n$ such that $V_{n-1}\supseteq U_n$ and $\alpha$ responds with an open set $V_n$ such that 
  $x_n\in V_n\subseteq U_n$.  Player $\alpha$ wins if 
$\bigcap\{V_n: n\in \omega\}\ne\emptyset$. Otherwise $\beta$ wins. We say that $(W_0,x_0,\ldots,W_k,x_k)$ is \textit{stronger} than $(U_0,y_0,\ldots, U_m,y_m)$ if
  $m\leq k\mbox{ and }U_0=W_0,\ldots,U_m=W_m\mbox{ and }x_0=y_0,\ldots,x_m=y_m.$
   We denote it by 
  $(U_0,y_0,\ldots, U_m,y_m)\preceq (W_0,x_0,\ldots, W_k,x_k)$. 
  We denote a sequence $(W_0,x_0,\ldots, W_k,x_k)$ by $(\overrightarrow{x}\circ \overrightarrow{W})$. 
  A topological space $X$ is called \textit{Choquet complete} if the player $\alpha$ has a winning strategy 
    in the strong Choquet game, denote it by $Ch(X)$.

For a topological space $X$, let $\tau(X)$ ($\tau^*(X)$) denote the topology on the set $X$ (the family of nonempty elements of $\tau(X)$).  A family $\mathcal P$ of open nonempty sets is called a
  \textit{$\pi$-base} if for every open nonempty set $U$ there is $P\in\mathcal P$ such that $P\subseteq U.$

A domain is a continuous directed complete partial order. This notion has been introduced by D. Scott as a model for the $\lambda$-calculus, for more information see \cite{arab}, \cite{scott}. Domain representable topological spaces were introduced by H. R. Bennett and D. J. Lutzer \cite{benluz}. We say that a topological space is domain representable if it is homeomorphic to the space of maximal elements of some continuous directed complete partial order topologized  with the Scott topology. 
In 2013 W. Fleissner and L. Yengulalp \cite{FY13} introduced  an equivalent definition of  a \textit{domain representable space} for $T_1$ topological spaces. We do not assume the 
antisymmetry condition on the relation $\ll$.
As suggested S.~\"{O}nal and \c{C}.~Vural in \cite{turcy} if we need additional antisymmetric property, let consider equivalent relation $E$ on the set $Q$ defined by ``$pEq$ if and only if ($p\ll q$ and $q\ll p$) or $p=q$''. We do not assume any separation axioms, if it is not explicitly stated.

We say that a topological space  $X$ is \textit{F-Y} (\textit{Fleissner--Yengulalp}) \textit{countably domain representable}   if there is  a triple $(Q,\ll, B)$ such that
 \begin{enumerate}
 \item[(D1)] $B: Q\to\tau^*(X)$ and $\{B(q):q\in Q\}$ is a base for $\tau(X)$,
 \item[(D2)] $\ll$ is a transitive relation on $Q$,
 \item[(D3)] for all $p,q\in Q$, $p\ll q$ implies $B(p)\supseteq B(q)$,
 \item[(D4)] For all $x\in X$, a set $\{q\in Q:x\in B(q)\}$ is upward directed by $\ll$, (every pair of elements has an upper bound),
 \item[(D5$_{\omega_1}$)] if $D\subseteq Q$ and $(D,\ll)$ is countable and upward directed, then $\bigcap\{B(q): q\in D\}\ne\emptyset$.
 \end{enumerate}
If the conditions (D1)--(D4) and a condition
 \begin{enumerate}
 \item[(D5)] if $D\subseteq Q$ and $(D,\ll)$ is  upward directed, then $\bigcap\{B(q): q\in D\}\ne\emptyset$
 \end{enumerate}
are satisfied, we say that a space $X$ is \textit{F-Y domain representable}.

W. Fleissner and L. Yengulalp in \cite{FY15} introduced a notion of a \textit{$\pi$-domain representable space}, 
this is analogous to the notion of a domain representable.
 
 We say that a topological space $X$ is \textit{F-Y} (\textit{Fleissner--Yengulalp}) \textit{countably $\pi$-domain  representable}   if there is a  triple $(Q,\ll, B)$ such that
 \begin{enumerate}
 \item[($\pi$D1)]$B: Q\to\tau^*(X)$ and $\{B(q):q\in Q\}$ is a $\pi$-base for $\tau(X)$,
 \item[($\pi$D2)] $\ll$ is a transitive relation on $Q$,
 \item[($\pi$D3)] for all $p,q\in Q$, $p\ll q$ implies $B(p)\supseteq B(q)$,
 \item[($\pi$D4)] if $q,p\in Q$ satisfy $B(q)\cap B(p)\ne\emptyset$, there exists $r\in Q$ satisfying $p,q\ll r$,
 \item[($\pi$D5$_{\omega_1}$)] if $D\subseteq Q$ and $(D,\ll)$ is countable and upward directed), then $\bigcap\{B(q): q\in D\}\ne\emptyset$.
 \end{enumerate}
 
 If the conditions ($\pi$D1)--($\pi$D4) and a condition
 \begin{enumerate}
 \item[($\pi$D5)] if $D\subseteq Q$ and $(D,\ll)$ is  upward directed, then $\bigcap\{B(q): q\in D\}\ne\emptyset$
 \end{enumerate}
are satisfied, we say that a space $X$ is \textit{F-Y $\pi$-domain representable}.

\section{$\pi$-domain representable spaces}
 
   P. S Kenderov and J. P. Revalski in \cite{ken-reval} have showed that the  set $E=\{f\in C(X):f\text{ attains  its  minimum  in }X\}$ contains $G_\delta$ dense subset of $C(X)$ is equivalent to existence of a winning strategy for $\alpha$ player in the Banach--Mazur game. Oxtoby \cite{ox} showed that if $X$ is metrizable space then  Player $\alpha$ has winning strategy in $BM(X)$  if and only if  $X$ contains a dense completely  metrizable subspace. A. Krawczyk and W. Kubi\'{s} \cite{kra-kub} have characterized the existence of winning strategies for player $\alpha$ in the abstract Banach--Mazur game played with finitely generated structures instead of open sets. In \cite{kub} there has been presented a version of the Banach--Mazur game played on  partially ordered set. We give a characterization of existence of a winning strategy for the player $\alpha$ in the Banach--Mazur game using the notation introduced by W. Fleissner and L. Yengulalp ``$\pi$-domain representable space''.
  
 \begin{theorem}\label{twierdz1}
A topological space $X$ is weakly $\alpha$-favorable if and only if $X$ is F-Y countably $\pi$-domain representable.
 \end{theorem}
  
 \begin{proof}
If $X$ is F-Y countably $\pi$-domain representable, then  it is easy to show that $X$ is weakly $\alpha$-favorable.

Assume that $X$ is weakly $\alpha$-favorable. We shall show that	$X$ is countably $\pi$-domain representable. 
 Let $s$ be a winning strategy for the player $\alpha$ in $BM(X)$. We consider a family $Q$ consisting of all finite sets
  $\{s(\overrightarrow{U}_0),\ldots,s(\overrightarrow{U}_i)\}$, 
  where $\overrightarrow{U}_m=(U^m_0,\ldots, U^m_{j_m}),\; m\leq i$  is a partial play, i.e.,
 $$U^m_0\supseteq s(U^m_0)\supseteq U^m_1\supseteq s(U^m_0,U^m_1)\supseteq\ldots\supseteq U^m_{j_m}
 \supseteq s(U^m_0,\ldots,U^m_{j_m})$$
  and $s(\overrightarrow{U}_0)\supseteq\ldots \supseteq s(\overrightarrow{U}_i)$. 
  
Let us define a relation $\ll$ on the family $Q$, 
  \begin{gather*}
 \{s(\overrightarrow{U}_0),\ldots ,s(\overrightarrow{U}_i)\} \ll \{s(\overrightarrow{W}_0),\ldots ,s(\overrightarrow{W}_k)\} \text{ iff } s(\overrightarrow{U}_i)\supseteq s(\overrightarrow{W}_0)  \ \& \\
  i\leq k \;\& \;\forall_{j\leq i}\;\exists_{r\leq k} \;\overrightarrow{U}_j\preceq \overrightarrow{W}_r.
  \end{gather*}
  Since $\preceq$ is transitive $\ll$ is transitive. 
  
  Let define a map $B:Q\to\tau^*(X)$ by the formula 
  $$B(\{s(\overrightarrow{U}_0),\ldots ,s(\overrightarrow{U}_i)\})=s(\overrightarrow{U}_i),$$
  for  $\{s(\overrightarrow{U}_0),\ldots ,s(\overrightarrow{U}_i)\}\in Q$. 
  
  Since $\{s(V):V\in\tau^*(X)\}$ is a $\pi$-base  $\{B(q):q\in Q\}$ is a $\pi$-base for $\tau$. It is easy to see that the map $B$ satisfies the condition $(\pi D3)$. 
   
   Towards item ($\pi$D4), let $p,q\in Q$ be such that $B(q)\cap B(p)\ne\emptyset$ and $p=\{s(\overrightarrow{U}_0),\ldots ,s(\overrightarrow{U}_i)\}$, $q=\{s(\overrightarrow{W}_0),\ldots ,s(\overrightarrow{W}_k)\}$ and  $s(\overrightarrow{U}_0)\supseteq\ldots\supseteq s(\overrightarrow{U}_i)$  and  $s(\overrightarrow{W}_0)\supseteq\ldots\supseteq s(\overrightarrow{W}_k)$. Since $V=B(p)\cap B(q)\subseteq
   s(\overrightarrow{U}_0)$ and $s$ is a winning strategy, we find an  element $\overrightarrow U'_0$ stronger  than $\overrightarrow{U}_0$ such that
   $s(\overrightarrow U'_0)\subseteq V$. Step by step we find a partial play $\overrightarrow U'_j$ such that
    $\overrightarrow{U}_j\preceq \overrightarrow U'_j$ and $s(\overrightarrow U'_j)\subseteq 
     s(\overrightarrow U'_{j-1})$ for $j\le i$. Since $s(\overrightarrow U'_i)\subseteq  s(\overrightarrow{W}_{0})$, 
     we find a partial play  $\overrightarrow W'_0$ such that $\overrightarrow{W}_0 \preceq \overrightarrow W'_0$ and
      $s(\overrightarrow W'_0)\subseteq s(\overrightarrow U'_i)$. Similarly, as for the sequence $p$, for the sequence
       $q$,  we define $s(\overrightarrow W'_l)$ with $\overrightarrow{W}_l\preceq
        \overrightarrow W'_l$ and $s(\overrightarrow W'_l)\subseteq  s(\overrightarrow W'_{l-1})$ for all $l\le k$. 
Continuing  in this way we get an element $r=\{s(\overrightarrow U'_0),\ldots ,s(\overrightarrow U'_i), s(\overrightarrow W'_0), \ldots ,s(\overrightarrow W'_k)\}$ such that $p,q\ll r$.

 Now we  show the condition ($\pi$D5$_{\omega_1}$). Let $D\subseteq Q$ be a countable upward directed set and let 
   $D=\{p_n: n\in \omega\}$. We define  a chain $\{q_n: n\in \omega\}\subseteq D\subseteq Q$ such that $p_n\ll q_n$ for $n\in \omega$. 
   By the condition ($\pi$D3), we get $\bigcap\{B(q_n):n\in \omega\}\subseteq\bigcap \{B(p):p\in D\}$. Each $q_n\in Q$ is of the form $q_n=\{s(\overrightarrow{W^n_0}),\ldots,
   s( \overrightarrow{W^n_{k_n}})\}.$
   
   Since $q_0\ll q_1$ there is $j\leq k_1$ such that $\overrightarrow{W^0_0}\preceq \overrightarrow{W}^1_j$. We have 
     $$s(\overrightarrow{W}^0_0)\supseteq B(q_0)=s(\overrightarrow{W}^0_{k_0})\supseteq s(\overrightarrow{W}^1_j)\supseteq B(q_1)=s(\overrightarrow{W}^1_{k_1}) .$$
 
Let $s(\overrightarrow U'_0)=s(\overrightarrow{W}^0_0)$ and $s(\overrightarrow U'_1)=s(\overrightarrow{W}^1_j)$. Inductively we can choose a sequence $\{s(\overrightarrow U'_n):n\in\omega\}$ such that $\overrightarrow{U}_n'
 \preceq \overrightarrow U'_{n+1}$ and 
 $$B(q_{n}) \supseteq s(\overrightarrow U'_{n+1})\supseteq B(q_{n+1}).$$ 
 Since $s$ is a winning strategy for the player $\alpha$, we have $$\emptyset\ne\bigcap
  \{s(\overrightarrow U'_n):n\in\omega\}=\bigcap\{B(q_n): n\in \omega\}\subseteq\bigcap\{B(p):p\in D\}.$$
 \end{proof}	   

 We give an example of a space, which is F-Y countably domain representable, but it is not F-Y $\pi$-domain representable.
 Note that this space will be  F-Y countably $\pi$-domain representable and not F-Y domain representable.
\begin{example}
We consider a space 
 $$X=\sigma\big(\{0,1\}^{\omega_1}\big)=\big\{x\in \{0,1\}^{\omega_1}: |\supp x|\le \omega\big\},$$ where $\supp x=\{\alpha \in A:x(\alpha)=1\}$ for $x\in \{0,1\}^A$, with the topology ($\omega_1$-box topology) generated by a base
 $$\mathcal{B}=\big\{\pr_A^{-1}(x): A\in [\omega_1]^{\le \omega}, x\in \{0,1\}^A\big \} ,$$
 where $\pr_A: \sigma(\{0,1\}^{\omega_1})\to \{0,1\}^A$ is a projection. 
 
We shall define a triple $(Q,\ll, B)$. Let $Q=\mathcal{B}$, a map $B:Q\to Q$ be the identity. Define a relation $\ll$ in the following way 
 $$\pr^{-1}_A(x_A)\ll\pr^{-1}_B(x_B)  \Leftrightarrow  \pr^{-1}_A(x_A)\supseteq \pr^{-1}_B(x_B),$$
 for any $\pr^{-1}_A(x_A), \pr^{-1}_B(x_B)\in \mathcal{B}$.
It is easy to see that the relation $\ll$ is transitive and it satisfies the condition $(D3)$. Now, we shall prove the condition (D4). Let $x\in X$ and $\pr^{-1}_{A_1}(x_{A_1}), \pr^{-1}_{A_2}(x_{A_2}) \in \{\pr^{-1}_A(x_A)\in \mathcal{B}:x\in \pr^{-1}_A(x_A)\}$. Since $x\in  \pr^{-1}_{A_1}(x_{A_1})\cap \pr^{-1}_{A_2}(x_{A_2})$ we get $x_{A_1}\upharpoonright A_2=x_{A_2}\upharpoonright A_1$. Set $A_3=A_1\cup A_2$ and $x_{A_3}\in \{0,1\}^{A_3}$ be such that $x_{A_3}\upharpoonright A_2=x_{A_2}$ and $x_{A_3}\upharpoonright A_1=x_{A_1}$. We have  $x_{A_3}\in \{0,1\}^{A_3}$ such that $x\in \pr^{-1}_{A_3}(x_{A_3})\subseteq \pr^{-1}_{A_1}(x_{A_1})\cap \pr^{-1}_{A_2}(x_{A_2})$. Hence $\pr^{-1}_{A_1}(x_{A_1}),\pr^{-1}_{A_2}(x_{A_2}) \ll \pr^{-1}_{A_3}(x_{A_3})$. 

We shall prove the condition $(D5)_{\omega_1}$. Let $D\subseteq \mathcal{B}$ be a countable upward directed family. We can construct a chain 
 $\{\pr^{-1}_{A_{n}}(x_{A_{n}}): n\in \omega\}\subseteq D$ such that for each set $\pr^{-1}_{A}(x_{A})\in D$ 
 there exists $n\in\omega$ such that $\pr^{-1}_{A}(x_{A})\ll\pr^{-1}_{A_{n}}(x_{A_{n}})$. 
 
Let $B=\bigcup\{A_n:n\in \omega\}$, $x_B\in \{0,1\}^B$ and $x_B\upharpoonright A_n=x_{A_n}$ for $n\in \omega$, then 
$$\bigcap\{\pr^{-1}_{A_n}(x_{A_n}): n\in \omega\}=\pr^{-1}_B(x_B)\in \mathcal{B},$$
and  $\pr^{-1}_B(x_B)\subseteq \bigcap D$, this completes the proof that the space $\sigma\big(\{0,1\}^{\omega_1}\big)$ is
F-Y countably domain representable.

Now we shall show that $X=\sigma\big(\{0,1\}^{\omega_1}\big)$ is not F-Y $\pi$-domain representable. 
Suppose that there exists the triple $(Q,\ll, B)$ satisfying the condition ($\pi$D1)--($\pi$D5). 
A family $\Pee =\{B(q):q\in Q\}$ is a $\pi$-base. 
By induction we define a sequence $\{Q_\alpha: \alpha<\omega_1\}$ such that the following conditions are satisfied:
\begin{enumerate}
\item $Q_\alpha\in [Q]^{\le \omega}$ and $Q_\alpha$ is upward directed, for $\alpha<\omega_1$,
\item $\bigcap \{B(q):q\in Q_\alpha\}=\pr^{-1}_{A_\alpha}(x_{A_\alpha}) \in \mathcal{B}$
for some $A_\alpha\in[\omega_1]^{\le \omega}$ and some $x_{A_\alpha}\in \{0,1\}^{A_\alpha}$, for $\alpha<\omega_1$,
\item $Q_\alpha\subseteq Q_\beta,$ for $\alpha<\beta<\omega_1$,
\item if $\bigcap \{B(q):q\in Q_\alpha\}=\pr^{-1}_{A_\alpha}(x_{A_\alpha})$ and $\bigcap \{B(q):q\in Q_\beta\}=\pr^{-1}_{A_\beta}(x_{A_\beta})$ for some $A_\alpha,A_\beta \in[\omega_1]^{\le \omega}$ and $x_{A_\alpha}\in \{0,1\}^{A_\alpha}$ and $x_{A_\beta}\in \{0,1\}^{A_\beta}$, then $\supp x_{A_\alpha}\varsubsetneq \supp x_{A_\beta}$, for  $\alpha<\beta<\omega_1$,
\end{enumerate}

We define a set $Q_0$. Take any  $r_0\in Q$. There exist a set $A_0\in [\omega_1]^{\le \omega}$ 
and $x_{A_0}\in \{0,1\}^{A_0}$ such that $\pr_{A_0}^{-1}(x_{A_0})\subseteq B(r_0)$. By conditions 
$(\pi D1), (\pi D3),(\pi D4)$ there exists $r_1\in Q$ such that $r_0\ll r_1$ and $B(r_1)\subseteq
 \pr_{A_0}^{-1}(x_{A_0}).$   Assume that we have defined $r_0\ll\ldots\ll r_n$ and a chain 
 $\{A_i:i\leq n\}\subseteq[\omega_1]^{\le \omega}$ and $x_{A_i}\in\{0,1\}^{A_i}$ such that
 $$\pr^{-1}_{A_{i-1}}(x_{A_{i-1}})\supseteq B(r_i)\supseteq \pr_{A_i}^{-1}(x_{A_i}) \text{ for } i\le n$$
 By conditions ($\pi$D1), ($\pi$D3), ($\pi$D4) there exists $r_{n+1}\in Q$ such that $r_n\ll r_{n+1}$ and $B(r_{n+1})\subseteq
 \pr_{A_{n}}^{-1}(x_{A_{n}}).$ There exist a set $A_{n+1}\in [\omega_1]^{\le \omega}$ and $x_{A_{n+1}}\in \{0,1\}^{A_{n+1}}$ such that $\pr_{A_{n+1}}^{-1}(x_{A_{n+1}})\subseteq B(r_{n+1})$. Let $Q_0=\{r_n:n\in\omega\}$. Then $\bigcap\{B(q):q\in Q_0\}=\bigcap\{\pr_{A_{n}}^{-1}(x_{A_{n}}):n\in\omega\}=\pr_{A}^{-1}(x_{A})$, where $A=\bigcup\{A_n:n\in\omega\}$ and $x_A$ is extension of the chain $\{x_{A_n}:n\in\omega\}$.

 Assume that we have defined  $\{Q_\alpha: \alpha< \beta\}$ which satisfies the conditions (D1)--(D4).
 
Let $\Raa _\beta= \bigcup \{Q_\alpha:\alpha< \beta\}.$ The set $\Raa _\beta$ is upward directed by conditions $(3), (1)$. Let $\Raa _\beta=\{p_n:n\in \omega\}$. By $(2)$ and $(3)$ we get $\bigcap\big\{B(p_n): n\in \omega \big\}\in\mathcal B$, hence there is a set $A_\beta\in [\omega_1]^{\le \omega}$ 
 and $x_{A_\beta}\in \{0,1\}^{A_\beta}$ such that $\bigcap \mathcal{R}_\beta=\pr_{A_\beta}^{-1}(x_{A_\beta})$. There exists  a set $A\in [\omega_1]^{\le \omega}$ 
  and $x_{A}\in \{0,1\}^{A}$ such that $\pr_{A}^{-1}(x_{A})\varsubsetneq \pr_{A_\beta}^{-1}(x_{A_\beta})$ and  $\supp x_{A_\beta}\varsubsetneq \supp x_{A}$. Since $\Pee $ is $\pi$-base we can find  $r_\beta\in Q$ such that $B(r_\beta)\subseteq \pr_{A}^{-1}(x_{A})$. Inductively we can define a sequence $\{q_n:n\in\omega\}\subseteq Q$, a chain  $\{A_n:n\in \omega\}\subseteq[\omega_1]^{\le \omega}$ and  a sequence $\{x_{A_n}\in\{0,1\}^{A_n}:n\in\omega\}$ such that $r_\beta, p_0\ll q_0$ and  $q_{n-1},p_n\ll q_n$ and 
 
  $$B(q_n)\supseteq \pr_{A_n}^{-1}(x_{A_n})\supseteq B(q_{n+1}) \text{ for } n\in\omega.$$

Let $Q_\beta=\bigcup \{Q_\alpha:\alpha< \beta\}\cup \{q_n:n\in \omega\}$. The set $Q_\beta$ satisfies conditions $(1)-(4)$, so we finish induction. The set
$\bigcup \{Q_\alpha:\alpha< \omega_1\}$ is upward directed.

By conditions $(2),(3)$ we have 
\begin{gather*}
\bigcap\{B(q):q\in \bigcup\{Q_\alpha:\alpha< \omega_1\}\big\}=\bigcap\{\pr_{A_\alpha}^{-1}(x_{A_\alpha}): \alpha <\omega_1\}=\\
=\pi_A^{-1}(x_A), \text{ for some } A \subseteq \omega_1\text{ and } x_A\in \{0,1\}^A,
\end{gather*}
where $\pi_A:\{0,1\}^{\omega_1}\to\{0,1\}^A$ is the projection.
By condition $(4)$ we get $|\supp x_A|=\omega_1$. Hence  $\pi_A^{-1}(x_A)\cap\sigma\big(\{0,1\}^{\omega_1}\big)=\emptyset$, a contradiction.
\begin{flushright}
$\Box$
\end{flushright}
 \end{example}
 
 Note that by the proof of Proposition 8.3. in \cite{FY15}, it follows that if there exists a triple $(Q, \ll, B)$, which F-Y countably $\pi$-represents a space $X$ and $|\bigcap\{B(q): q\in D\}|=1$ for every countable and upward directed set  $D\subseteq Q$, then this triple F-Y $\pi$-represents a space $X$.
 
\begin{theorem}
The Cartesian product of any family of F-Y countably $\pi$-domain representable spaces is F-Y countably $\pi$-domain representable.
\end{theorem} 
\begin{proof}
Let $X$ be a product of a family $\{X_a:a\in A\}$ of F-Y countably $\pi$-domain representable spaces.
Let $(Q_a,\ll_a, B_a)$ be a  triple which satisfies conditions ($\pi$D1)--($\pi$D4) and ($\pi$D5$_{\omega_1}$).
Any basic  nonempty open subset $U$ in $X$ is of the form $U=\prod\{U_a:a\in A\}$ where $U_a$ is nonempty
open subset of $X_a$ and $U_a=X_a$ for all but a finite number of $a$. We may assume that $0\in Q_a$ is the least element in $Q_a$  and $B_a(0)=X_a$ for each $a\in A$. Put 
$$Q=\left\{p\in\prod\{Q_a:a\in A\}:|\{a\in A:p(a)\ne 0\}|<\omega\right\}.$$ 

 Define a relation $\ll$ in $Q$ by the formula
$$p\ll q\Longleftrightarrow p(a)\ll_a q(a) \text{ for all } a\in A,$$
where $p,q\in Q$. Let us define a map $B:Q\to\tau^*(X)$ by $B(p)=\prod\{B_a(p(a)):a\in A\}$, where $p\in Q$. It is easy to check that $(Q,\ll,B)$ is countably $\pi$-domain represents $X$. 
\end{proof}
 
 \section{Domain representable spaces}
 
 In 2003 K. Martin  \cite{m03} showed that, if a space is domain representable, then the player $\alpha$ has a winning strategy in the strong Choquet game. In 2005 W. Fleissner and L.  Yengulalp \cite{FY15} showed that it is sufficient that a space is countably domain representable. Now, we shall show that the property of being countably domain representable is necessary.
 
 \begin{theorem}
 A topological space $X$ is Choquet complete if and only if it is F-Y countably domain representable.
\end{theorem}
\begin{proof}
By Theorem 4.3 (3) in \cite{FY15} (see also \cite{m03}) it suffices to prove that if $X$ is Choquet complete, then $X$ is F-Y countably domain representable.
Assume that $X$ is Choquet complete.

 Let $s$ be a winning strategy for player $\alpha$. We consider a family $Q$ consisting of all finite sets
  $\{s(\overrightarrow{x_0}\circ\overrightarrow{U}_0),\ldots ,s(\overrightarrow{x_i}\circ\overrightarrow{U}_i)\}$, 
  where $\overrightarrow{x_m}\circ\overrightarrow{U}_m=(U^m_0,x^m_0,\ldots, U^m_{j_m},x^m_{j_m}),\; m\leq i$,  is a partial play in the strong Choquet game, i.e.,
\begin{gather*}
U^m_0\supseteq s(U^m_0,x^m_0)\supseteq U^m_1\supseteq s(U^m_0,x^m_0,U^m_1,x^m_1)\supseteq\ldots\supseteq U^m_{j_m}\\
 \supseteq s(U^m_0,x^m_0,\ldots,U^m_{j_m},x^m_{j_m})
\end{gather*} 
  and $s(\overrightarrow{x_0}\circ\overrightarrow{U}_0)\supseteq\ldots \supseteq s(\overrightarrow{x_i}\circ\overrightarrow{U}_i)$.
  
  Let us define a relation $\ll$ on the family $Q$, 
  \begin{gather*}
 \{s(\overrightarrow{x_0}\circ\overrightarrow{U}_0),\ldots ,s(\overrightarrow{x_i}\circ\overrightarrow{U}_i)\}\ll \{s(\overrightarrow{y_0}\circ\overrightarrow{W}_0),\ldots ,s(\overrightarrow{y_k}\circ\overrightarrow{W}_k)\} \text{ iff }\\ s(\overrightarrow{x_i}\circ\overrightarrow{U}_i)\supseteq s(\overrightarrow{y_0}\circ\overrightarrow{W}_0) \; \&  \;
  i\leq k \; \& \;\forall_{j\leq i}\;\exists_{r\leq k} \;(\overrightarrow{x_j}\circ\overrightarrow{U}_j)\preceq (\overrightarrow{y_r}\circ\overrightarrow{W}_r).
  \end{gather*} 
  
  We define a map $B:Q\to\tau^*$ by the formula 
  $$B\{s(\overrightarrow{x_0}\circ\overrightarrow{U}_0),\ldots ,s(\overrightarrow{x_i}\circ \overrightarrow{U}_i)\}=s(\overrightarrow{x_i}\circ\overrightarrow{U}_i),$$
 with $s(\overrightarrow{x_i} \circ \overrightarrow{U}_i)\subseteq s(\overrightarrow{x_j}\circ \overrightarrow{U}_j)$  for $j\le i$, for each $\{s(\overrightarrow{x_0}\circ \overrightarrow{U}_0),\ldots ,s(\overrightarrow{x_i}\circ\overrightarrow{U}_i)\}\in Q$. The rest of the proof is similar to the proof of Theorem \ref{twierdz1}.
\end{proof}


\begin{thebibliography}{40}

\bibitem{arab} {\sc S.~Abramsky, A.~Jung}, {\it Domain theory}, in: S. Abramsky, D.M. Gabbay, T.S.E. Maibaum (Eds.), Handbook of Logic in Computer Science, vol. III, Oxford University Press, Oxford, 1994.


\bibitem{benluz}{\sc H. Bennett, D. Lutzer}, {\it Domain representable space}, Fund. Math. 189 (3) (2006) 255--268.

\bibitem{FY13}{\sc W.~Fleissner, L.~Yengulalp}, {\it When $C_p(X)$ is Domain Representable}, Fund. Math. 223 (1) (2013) 65--81.

\bibitem{FY15}{\sc W.~Fleissner, L.~Yengulalp}, {\it From subcompact to domain representable}, Topol. Appl. 195 (2015) 174--195.

\bibitem{ken-reval}  {\sc P.~S.~Kenderov, J.~P.~ Revalski}, {\it The Banach-Mazur game and generic existence of solutions to optimization problems}, Proc. Amer. Math. Soc. 118 (1993), no. 3, 911--917.

\bibitem{kra-kub} {\sc A.~Krawczyk, W.~ Kubi\'{s}}, {\it Games on finitely generated structures}, (arXiv:1701.05756).

\bibitem{kub} {\sc W.~Kubi\'{s}}, {\it Banach-Mazur game played in partially ordered sets}, Banach Center Publications 108 (2016) 151--160, (arXiv:1505.01094).

\bibitem{m03}{\sc K.~Martin}, {\it Topological games in domain theory}, Topol. Appl. 129  (2003) 177--186.

\bibitem{ox} {\sc J.~C.~Oxtoby}, {\it The Banach-Mazur game and Banach Category Theorem}, Contributions to the Theory of Games, Vol. III, Ann. of Math. Stud., no. 39, Princeton Univ. Press, Princeton, NJ, 1957,  159--163.

\bibitem{scott} {\sc D. Scott}, {Outline of a mathematical theory of computation}, Technical Monograph PRG--2, November 1970.
 
\bibitem{turcy} {\sc S. \"Onal, \c{C}. Vural}, {\it Domain representability of retract}, Topol. Appl 216 (2017) 79--84.

\bibitem{Telgarsky} {\sc R. Telg\'arsky}, {\it Topological games: on the 50th anniversary of the Banach-Mazur game},
Rocky Mountain J. Math. 17 (1987) 227--276.

\bibitem{White}{\sc H.~E.~ White, Jr.},
 \textit{Topological spaces that are $\alpha$-favorable for player with perfect information}, Proc. Amer. 
 Math. Soc. 50 (1975), 477--482.
 
\end{thebibliography}
\end{document}